\documentclass{amsart}
\usepackage{amsmath}
\usepackage{amsthm}
\usepackage{amssymb}
\usepackage{tikz}
\usetikzlibrary{automata,positioning}
\usepackage{enumitem}
\usepackage{mathrsfs}
\usepackage{scalerel}

\usepackage{hyperref}

\newtheorem{theorem}{Theorem}[section]

\newtheorem{alphatheorem}{Theorem}

\theoremstyle{definition}

\newtheorem{definition}[theorem]{Definition}
\newtheorem{proposition}[theorem]{Proposition}

\newtheorem*{proposition*}{Proposition}
\newtheorem*{observation*}{Observation}
\newtheorem*{claim*}{Claim}

\newtheorem*{lemma*}{Lemma}

\newtheorem{corollary}[theorem]{Corollary}
\newtheorem{remark}[theorem]{Remark}

\newtheorem*{conjecture*}{Conjecture}

\newtheorem*{convention*}{Convention}

\theoremstyle{plain}
\newtheorem{lemma}[theorem]{Lemma}

\newcommand{\bra}[1]{ \left( #1 \right) }

\newcommand{\abs}[1]{\left|#1\right|}

\newcommand{\norm}[1]{\left\lVert #1 \right\rVert}

\newcommand{\bbra}[1]{ { \left\{ #1 \right\} } }

\newcommand{\cP}{\mathcal{P}}

\newcommand{\cN}{\mathcal{N}}
\newcommand{\cW}{\mathcal{W}}

\newcommand{\e}{\varepsilon}

\usepackage{mathrsfs}
\usepackage[mathscr]{euscript}

\newcommand{\NN}{\mathbb{N}}
\newcommand{\QQ}{\mathbb{Q}}

\newcommand{\PP}{\mathbb{P}}

\DeclareMathOperator*{\EE}{\scalerel*{\mathbb{E}}{\big(}}

\newcommand{\ZZ}{\mathbb{Z}}
\newcommand{\RR}{\mathbb{R}}
\newcommand{\R}{\RR}

\newcommand{\CC}{\mathbb{C}}

\newcommand{\floor}[1]{\left\lfloor #1 \right\rfloor}

\newcommand{\set}[2]{\left\{ #1 \ \middle| \ #2 \right\} }

\newcommand{\barbb}[1]{\bar{ \bb #1}}
\newcommand{\ffrac}[2]{#1/#2}

\newcommand{\parbreak}[1]{
\begin{center}
***
\end{center}
}

\newcommand{\bb}{\mathbf}

\newcommand{\comment}[1]{}

\usepackage{soul}

\usepackage{stmaryrd}


\begin{document}

\author[ J. Konieczny]{   Jakub Konieczny
}
\address{Mathematical Institute \\ 
University of Oxford\\
Andrew Wiles Building \\
Radcliffe Observatory Quarter\\
Woodstock Road\\
Oxford\\
OX2 6GG}
\email{jakub.konieczny@gmail.com}

\title{Gowers norms for the Thue-Morse and Rudin-Shapiro sequences}

\maketitle

\begin{abstract}
	We estimate Gowers uniformity norms for some classical automatic sequences, such as the Thue-Morse and Rudin-Shapiro sequences. The methods are quite robust and can be extended to a broader class of sequences.

As an application, we asymptotically count arithmetic progressions of a given length in the set of integers $\leq N$ where the Thue-Morse (resp.\ Rudin-Shapiro) sequence takes the value $+1$.
\end{abstract}

\newcommand{\Qs}{\{0,1\}^s}
\newcommand{\cL}{\mathcal{L}}
\section{Introduction}

The Thue-Morse sequence is among the simplest automatic sequences. It can be described by the recursive relations: 
$$
	t(0) = 1, \qquad t(2n) = t(n), \qquad t(2n+1) = -t(n),
$$
or by the explicit formula $t(n) = (-1)^{s_2(n)}$, where $s_2(n)$ denotes the sum of digits of $n$ base $2$. Arguably, the Thue-Morse sequence is very structured --- in particular, its subword complexity (i.e.\ number of distinct subsequences of a given lenght) has linear rate of growth  (this is a general feature of automatic sequences). On the other hand, there are also ways in which it can be construed as pseudorandom. 

Mauduit and Sark\"{o}zy \cite{Mauduit-1998} studied several measures of pseudorandomness for the Thue-Morse sequence, and showed that $t(n)$ is highly uniform according to some but not all of those measures. In particular, it is shown that for any positive integers $a,b,M,N$ with $a (M-1) + b < N$, we have
\begin{equation}
	\label{eq:cor-ap-TM}
	{ \sum_{n=0}^{M-1}  t(a n + b) } = O(N^{\log 3/ \log 4}),
\end{equation}
where the implied constant is absolute. In fact, this easily follows from the bound obtained by Gelfond \cite{Gelfond-1967}:
\begin{equation}
	\label{eq:cor-lin-TM}
	\sup_{\alpha \in \RR} \abs{ \sum_{n=0}^{N-1} t(n) e(\alpha n)} = O(N^{\log 3/ \log 4}),
\end{equation}
where as usual $e(x) := e^{2\pi i x}$. However, for all sufficiently large integers $N$, there exist positive integers $M$ and $h$ with $M+ h \leq N $ such that
\begin{equation}
	\label{eq:cor-self-TM}
	\abs{ \sum_{n=0}^{M-1} t(n)t(n+h) } \geq c N,
\end{equation}
where $c > 0$ is an absolute constant. (We may take $c = \ffrac{1}{12}$ and $h = 1$.)
Thus, the Thue-Morse sequence does not correlate with arithmetic progressions but can have some large self-correlations.

In a different direction, $t(n)$ is believed to look highly random when restricted to certain subsequences. Several conjectures to this effect, in a slightly more general situation, are known as the Gelfond Problems \cite{Gelfond-1967}. 

Let $\cP$ denote the set of prime numbers and $\pi(N) = \abs{ \cP \cap [N]}$. Gelfond conjectured that it should hold that
\begin{equation}
	\label{eq:Gelfond1-TM}
	\abs{ \set{ p \in [N] \cap \cP }{ t(p) = +1 } } = \frac{1}{2}\pi(N) + O(N^{1-c}),
\end{equation}
where $c > 0$ is an absolute constant. This was proved only recently by Mauduit and Rivat \cite{MauduitRivat-2010}.

Let $p(x) \in \QQ[x]$ be a polynomial with $p(\ZZ) \subset \ZZ$, and extend $t(n)$ to $\ZZ$ by putting $t(-n) = t(n)$. Another of Gelfond's conjectures asserts that we should have
\begin{equation}
	\label{eq:Gelfond2-TM}
	\abs{ \set{ n \in [N] }{ t(p(n)) = +1 } } = \frac{1}{2}N + O(N^{1-c}),
\end{equation}
for an absolute constant $c > 0$. This is only known for polynomials of degree $2$ by work of Mauduit and Rivat \cite{MauduitRivat-2009}. In fact, for $p(n) = n^2$, a much stronger result is shown in \cite{DrmotaMauduitRivat-TM-squares}, implying in particular that $t(n^2)$ is normal (i.e. that each block of $\pm 1$'s of length $l$ appears in $t(n^2)$ with frequency $\ffrac{1}{2^l}$). Namely, for each $l$ and each $\epsilon_0,\dots, \epsilon_{l-1} \in \{-1,+1\}$ we have
\begin{equation}
	\label{eq:Gelfond22-TM}
	\abs{ \set{ n \in [N] }{ t((n+i)^2) = \epsilon_i \text{ for } 0 \leq i < l } } = \frac{1}{2^l}N + O(N^{1-c}),
\end{equation}
where the constant $c$ depends only on $l$.

Finally, let $\alpha \in \RR_{>0} \setminus \ZZ$. It is conjectured (see e.g.\ \cite{Drmota-2014}) that the sequence $t(\floor{ n^\alpha})$ should be normal for any such $\alpha$. This is confirmed in the quantitative sense by M\"{u}llner and Spiegelhofer \cite{MullnerSpiegelhofer}, who showed that for any $1 < \alpha < 3/2$, for any $l$ and $\e_i \in \{-1,+1\}$ for $0 \leq i < l$ it holds that 
\begin{equation}
	\label{eq:GelfondX-TM}
	\abs{ \set{ n \in [N] }{ t(\floor{(n+i)^\alpha}) = \epsilon_i \text{ for } 0 \leq i < l } } = \frac{1}{2^l}N + O(N^{1-c}),
\end{equation}
where the constant $c$ depends on $l$ and $\alpha$.

In general, it is believed that $t(n)$ restricted to any of the aforementioned sequences should be a normal sequence (we refer e.g.\ to \cite{Drmota-2014}, which is also an excellent reference for related results).

Here, we consider a different notion of pseudorandomness related to Gowers uniformity norms. First introduced by Gowers in his work on an new proof of Szemeredi's theorem \cite{Gowers-2001}, these norms now play a crucial role in additive combinatorics. For exposition of the relevant theory, we refer to \cite{Green-book} and \cite{Tao-book}.

\begin{definition}
	Fix $s \in \NN$. For $N \in \NN$ and $f \colon [N] \to \RR$, the $s$-th Gowers uniformity norm of $f$ is defined by
	\begin{equation}
		\norm{f}_{U^s[N]}^{2^s} := \EE_{\substack{ n, \bb h }} \prod_{\omega \in \Qs} f (n + \omega \cdot \bb h)
	\label{eq:def-Gowers}
	\end{equation}	  
	where $\omega \cdot \bb h = \sum_{i = 1}^s \omega_i h_i$, and the expectation is taken over all $n \in \ZZ,\ \bb h \in \ZZ^s$ for which the cube $\set{ n + \omega \cdot \bb h }{ \omega \in \Qs } $ is contained in $[N]$.  
\end{definition}

A sequence is (informally) said to be Gowers uniform of order $s$ if its $U^s[N]$-norm is small. We show that $t(n)$ indeed is highly Gowers uniform, which makes it one of the simplest sequences known to be Gowers uniform of all orders.

\begin{alphatheorem}\label{thm:TM}
	Let $t\colon \NN_0 \to \{\pm 1\}$ denote the Thue-Morse sequence. For any $s \in \NN$, there exists $c = c(s) > 0$ such that $\norm{t}_{U^s[N]} = O(N^{-c})$ as $N \to \infty$.
\end{alphatheorem}

A key reason for interest in the Gowers uniformity norms is their usefulness in counting linear patterns. In particular, as an immediate corrolary of Theorem \ref{thm:TM}  we conclude via the Generalised von Neumann Theorem that the number of $k$-term arithmetic progressions in $\set{n \in [N]}{t(n) = +1}$ is $N^2/2k + O(N^{2-c})$. 

We also remark that sequences with small Gowers norms do not correlate with polynomials phases: if $p \in \RR[x]$ with $\deg p = s-1$ then $ \EE_{n < N} f(n)e(p(n)) \ll \norm{f}_{U^s[N]}.$ 
Hence, Theorem \ref{thm:TM} implies that $t(n)$ is a fully oscillating sequence in the terminology of \cite{Fan-A}. As an application, for any dynamical system with quasi-discrete spectrum $(X,T)$ and any $f_1,\dots,f_l \in C(X)$, $q_1,\dots,q_l \in \QQ[x]$ with $q_i(\NN_0) \subset \NN_0$ for $1 \leq i \leq l$, we have  
\begin{equation}
\label{eq:cor-of-Fan}
	\lim_{N \to \infty} \EE_{n < N} t(n) \prod_{i=1}^l f_i(T^{q_i(n)} x) = 0
\end{equation}
for any point $x \in X$.

A subtly different type of uniformity norms $\norm{ \cdot }_{U(s)}$ on $l^\infty(\ZZ)$ is introduced and studied in \cite{HostKra-2009}. In fact, it follows from results obtained there that $\norm{t}_{U(s)} = 0$ for all $s \in \NN$ (cf.\ Proposition 2.21), and Theorem \ref{thm:TM} can be construed as an analogue of this result. As an example application (cf.\ remarks after Theorem 2.25), this implies that for any measure preserving system $(X,T,\mu)$, any $f_1,\dots,f_l \in L^\infty(X)$ we have
\begin{equation}
	\EE_{n < N} t(n) \prod_{i=1}^l f_i(T^{in} x) \to 0 \quad \text{ in $L^2(\mu)$ as $N \to \infty$.}
\end{equation}

\mbox{} 

A slightly more complicated sequence we deal with carries the name of Rudi-Shapiro. It is recursively given by
$$
	r(0) = 0,\ r(2n) = r(n),\ r(4n+1) = r(n),\ r(4n+3) = -r(2n+1),
$$
or explicitly by $r(n) = (-1)^{f_{\mathtt{11}}(n) }$, where $f_{\mathtt{11}}(n)$ denotes the number of times the pattern $\mathtt{11}$ appears in the binary expansion of $n$.

Much like in the case of the Thue-Morse sequence, various pseudorandomness properties of the Rudin-Shapiro sequence have long been studied. In \cite{Mauduit-1998} it is shown that $r(n)$ does not correlate with arithmetic progressions, but has large self-correlations. More precisely, we have 
\begin{equation}
	\label{eq:cor-ap-RS}
	{ \sum_{n=0}^{M-1}  t(a n + b) } = O(N^{1/2}),
\end{equation}
for any $a,b,M$ with $a(M-1) +b<N$, while there exist $M,h$ with $M+h \leq N$ such that
\begin{equation}
	\label{eq:cor-self-RS}
	\abs{ \sum_{n=0}^{M-1} t(n)t(n+h) } \geq c N,
\end{equation}
where one can take $c = 1/6$ if $N$ is sufficiently large.

There are also exist results asserting pseudorandomness of the restrictions of $r(n)$ to certain subsequences. In the case of the primes, in analogy with \eqref{eq:Gelfond1-TM}, we have
\begin{equation}
	\label{eq:Gelfond1-RS}
	\abs{ \set{ p \in [N] \cap \cP }{ t(p) = +1 } } = \frac{1}{2}\pi(N) + O(N^{1-c}),
\end{equation}
where $c > 0$ is an absolute constant \cite{MauduitRivat-2015}. Likewise, in analogy to \eqref{eq:GelfondX-TM}, it follows as a special case from results in \cite{DeshouillersDrmotaMorgenbesser-2012} that
\begin{equation}
	\label{eq:GelfondX-RS}
	\abs{ \set{ n \in [N] }{ r(\floor{ n^\alpha}) = +1 } } = \frac{1}{2}N + o(N).
\end{equation}
for any $1 < \alpha < 7/5$. (It is not known if $r(\floor{ n^\alpha})$ is normal.)

We show that $r(n)$ is highly Gowers uniform. As an application we conclude that the number of $k$-term arithmetic progressions in $\set{n \in [N]}{r(n) = +1}$ is $N^2/2k + O(N^{2-c})$.

\begin{alphatheorem}\label{thm:RS}
	Let $r\colon \NN_0 \to \{\pm 1\}$ denote the Rudin-Shapiro sequence. For any $s \in \NN$, there exists $c = c(s) > 0$ such that $\norm{r}_{U^s[N]} = O(N^{-c})$ as $N \to \infty$.
\end{alphatheorem}

While we focus our attention on these two specific sequences, many of the observations apply to more general automatic sequences. A sequence $a \colon \NN_0 \to \CC$ is \emph{$k$-automatic} if $a(n)$ can be computed by a finite automaton taking $k$-ary expansion of $n$ on input. For comprehensive background, we refer to \cite{AS}.

\subsection*{Notation} We write $\NN = \{1,2,\dots\}$ and $\NN_0 = \NN \cup \{0\}$. We use standard asymptotic notation. For any expressions $X$, $Y$, we write $X = O(Y)$ or $X \ll Y$ if there exists a constant $c > 0$ such that $X < cY$. We consistently use boldface letter $\bb x$ to denote vector with coordinates $(x_i)$; and also write $\abs{\bb x} := \sum_i \abs{x_i}$. By $[N]$ we denote the interval $\{0,1,\dots,N-1\}$.

\subsection*{Acknowledgements} The author is grateful to Tanja Eisner for her hospitality during his stay in Leipzig when the work on this project began; to Jakub Byszewski for many long and productive discussions; to Christian Mauduit, Clemens M\"{u}llner, and Aihua Fan for helpful comments; and to Ben Green for his encouragement and valuable advice. 

\section{Thue-Morse sequence}\label{sec:TM}

The purpose of this section is to prove Theorem \ref{thm:TM}, asserting that the uniformity norms of the Thue-Morse sequence $t(n) = (-1)^{s_2(n)}$ are small. Throughout this section, let $s \in \NN_{\geq 2}$ be fixed. (We may assume that $s \geq 2$ since $\norm{t}_{U^2[N]} \gg \norm{t}_{U^1[N]}$.) It will be convenient to  study somewhat more general averages 
\begin{equation}
	\label{eq:def-A-TM}
	A(L, \bb r) := \EE_{n, \bb h }\prod_{\omega \in \Qs} t( n + \omega \cdot \bb h + r_\omega),
\end{equation}
where $n, \bb h$ parametrize the cubes $\set{ n + \omega \cdot \bb h}{\omega \in \Qs} \subset [2^L]$, $L \in \NN_0$ and $\bb r = (r_\omega)_{\omega \in \Qs}$ with $r_\omega \in \ZZ$. Note that if $\bb r = \bb 0$, then \eqref{eq:def-A-TM} defines $\norm{r}_{U^s[2^L]}^{2^s}$.

\begin{lemma}\label{lem:recurrence-A-TM}
	The averages $A(L, \bb r)$ satisfy the recursive relation
	\begin{equation}
		\label{eq:recurrence-A-TM}
		A(L, \bb r) = (-1)^{\abs{ \bb r}} \EE_{ \bb e } A(L-1, \delta(\bb r; \bb e)) + O(2^{-L}),
	\end{equation}
	where the average is taken over $\bb e = (e_i)_{i=0}^s \in \{0,1\}^{s+1}$ and $\delta(\bb r; \bb e)$ is given by 
\begin{equation}
\label{eq:005}
	\delta(\bb r; \bb e)_\omega = \floor{ \frac{ r_\omega + (1,\omega) \cdot \bb e  }{2} } = \floor{ \frac{ r_\omega + e_0 + \sum_{i=1}^s \omega_i e_i }{2} }.
	\end{equation}	

\end{lemma}
\begin{proof}
	For any cube $Q = \set{ n + \omega \cdot \bb h }{ \omega \in \Qs }$, there exists a choice of a unique cube $Q' = \set{ n' + \omega \cdot \bb h' }{ \omega \in \Qs }$ and a vector $\bb e \in \{0,1\}^s$ such that $Q = \set{ 2 (n' + \omega \cdot \bb h') + (1,\omega) \cdot \bb e }{ \omega \in \Qs }$; we simply take $n = 2n' + e_0$ and $h_i = 2h_i + e_i$. Moreover, if $Q \subset [2^L]$ is chosen uniformly at random, then $Q' \subset [2^{L-1}]$ with probability $1 - O(2^{-L})$, and conversely if $Q' \subset [2^{L-1}]$ and $\bb e \in \{0,1\}^s$ are chosen uniformly at random then $Q \subset [2^{L}]$ with probability $1 - O(2^{-L})$.
It follows that 
\begin{align*}
		A(L, \bb r) &= 
		\EE_{ \bb e } \EE_{n', \bb h' } \prod_{\omega \in \Qs} t( 2 n' + 2\omega \cdot \bb h' + (1,\omega) \cdot \bb e + r_\omega ) + O(2^{-L}) 
		\\& = \EE_{ \bb e } \EE_{n', \bb h' } \prod_{\omega \in \Qs} (-1)^{(1,\omega) \cdot \bb e + r_\omega } t( n' + \omega \cdot \bb h' + \delta(\bb r; \bb e)_\omega)  + O(2^{-L})
		\\& = \EE_{ \bb e } (-1)^{S(\bb e)} A(L-1,\delta(\bb r; \bb e)) + O(2^{-L}),
\end{align*}
where $\delta(\bb r; \bb e)$ is defined by \eqref{eq:005} and $S(\bb e) = \sum_\omega r_\omega + 2^s e_0 +  2^{s-1} \sum_{i=1}^s e_i\equiv \abs{\bb r} \bmod{2}$ (expectation is over $\bb e \in \{0,1\}^s$ and $\set{n' +\omega\cdot \bb h'}{\omega \in \{0,1\}^s } \subset [2^{L-1}]$).
\end{proof}

\renewcommand{\R}{\mathcal{R} }

Lemma \ref{lem:recurrence-A-TM} motivates us to introduce a random walk $\cW_{\mathrm{TM}}$ on a directed graph $G = (V,E)$ defined as follows. The set of vertives is $V = V_+ \cup V_-$ where $V_\pm = \set{ (\bb r, \pm 1) }{ \bb r \in \ZZ^{\Qs} }$. The transition probabilities are given by
\begin{equation}
	\label{eq:deg-P-TM}
	P\bra{ (\bb r, \pm 1) ; ( \bb r', \pm (-1)^{\abs{ \bb r}}  } 
	= 
	\PP_{\bb e }\bra{ 
\delta(\bb r; \bb e) = \bb r'
},
\end{equation}
where $\bb e = (e_i)_{i=0}^s$ is uniformly distributed in $\bbra{0,1}^{s+1}$ and $\delta(\bb r; \bb e)$ is given by \eqref{eq:005}. (By convention, the two occurrences of the symbol $\pm$ both denote the same sign.) The remaining transition probabilities (i.e. those where the signs do not agree) are declared to be identically $0$.

The set $E$ of (directed) edges of $G$ consists of the pairs $(v,v') \in V^2$ with $P(v,v') > 0$; hence the edge $(\bb r, \pm 1) \to (\bb r', \pm (-1)^{\abs{\bb r}})$ is present if and only if there exists $\bb e \in \{0,1\}^{s+1}$ such that $\delta(\bb r; \bb e) = \bb r'$ (with $\delta(\bb r; \bb e)$ given by \eqref{eq:005}). We will be particularly interested in the graph $G_0$ supported on the vertices $V_0$ reachable from the initial vertex $v_0 = (\bb 0, +1)$.

We note that $\cW_{\mathrm{TM}}$ comes with a natural symmetry $\R \colon V \to V$ given by $(\bb r, \pm 1) \mapsto (\bb r, \mp 1)$. We have $\R ( \R( v ) ) = v $ and $P( \R(v), \R(v') ) = P( v, v')$ for all $v,v' \in V$. In particular, $\R$ preserves the edges of $G$.

Denote further by $P^{(l)}(v,v')$ the probability of reaching vertex $v'$ after $l$ steps, starting from $v$. Iterating Lemma \ref{lem:recurrence-A-TM} we obtain the following formula. 

\begin{corollary}\label{cor:recurrence-A-TM}
	The averages $A(L, \bb r)$ satisfy for any $l < L$ the recursive relation
	\begin{align}
		\label{eq:recurrence-A-TM-IIZ}
		A(L, \bb r) = \sum_{ \bb r', \sigma } 
		P^{(l)} \big( (\bb r, +1), (\bb r',\sigma) \big) \sigma A(L-l, \bb r' )
		+ O(2^{-(L-l)}),
	\end{align}
	where the sum runs over all pairs $(\bb r',\sigma) \in V$ which are reachable from $(\bb r, +1)$.
	In particular,
	\begin{align}
		\label{eq:recurrence-A-TM-II}
		\norm{ t }_{U^s[N]}^{2^s} = \sum_{ \bb r' } 
		 \bra{ P^{(l)}\bra{ v_0, (\bb r',+1) } - P^{(l)}\bra{ v_0, (\bb r',-1) } }  A(L-l, \bb r') + O(2^{-(L-l)}),
	\end{align}
	where the sum runs over all $\bb r'$ such that at least one of $(\bb r', \pm 1)$ belongs to $V_0$.
\end{corollary}
\begin{proof}
	Apply Lemma \ref{lem:recurrence-A-TM} $l$ times. Note that the recurrence relation \eqref{eq:recurrence-A-TM} can be equivalently written as
	$$
	A(L, \bb r) = \sum_{ \bb r', \sigma } 
		P \big( (\bb r, +1), (\bb r',\sigma) \big) \sigma A(L-1, \bb r' )
		+ O(2^{-L}),
	$$
	which is the same as \eqref{eq:recurrence-A-TM-IIZ} for $l = 1$. For general $l \in \NN$, the main term follows directly from how $P^{(l)}(v,v')$ are defined.
	The total error term is $\ll 2^{-L} + 2^{(-L-1)} + \dots + 2^{-(L-l)} \ll 2^{-(L-l)}$. 
\end{proof}

Recall that a directed graph is strongly connected if there exists a directed path from any vertex to any other vertex. A graph is aperiodic if the greatest common divisor of all cycles present in the graph equals $1$.

\begin{proposition}\label{lem:G-TM-connected}
	Let $G_0$ be the graph constructed above. Then $G_0$ is finite, strongly connected, aperiodic, and preserved by $\R$.
\end{proposition}
\begin{proof}
	Aperiodicity follows immediately from the observation that $G_0$ contains a loop at $(\bb 0, +1)$
		since $\delta((\bb 0, +1); \bb 0) = (\bb 0, +1)$.
	
	If $G$ contains an edge from $v = (\bb r, \pm 1)$ to $v' = (\bb r', \pm 1)$  then $r'_\omega = \delta(\bb r; \bb e) = \floor{ \frac{r_\omega + (1,\omega) \cdot \bb e }{2}}$ for some $\bb e \in \{0,1\}^{s+1}$, and in particular $0 \leq (1,\omega) \cdot \bb e \leq \abs{\omega}+1$. An elementary inductive argument now shows that if $(\bb r,\pm 1) \in V_0$ then $0 \leq r_\omega \leq \abs{\omega}$, which proves finiteness.

	Similarly, taking $\bb e = \bb 0 \in \{0,1\}^{s+1}$, we see that any vertex $v = (\bb r, \pm 1)$ has an edge to some $v' = (\bb r', \pm 1)$ with $\abs{ \bb r'} \leq \abs{ \bb r }/2$. Repeating this argument, we may find a path from any $v \in V_0$ to one of $(\bb 0, +1),\ (\bb 0, -1)$. Thus, to prove that $G_0$ is strongly connected, it will suffice to show that there exists a path from $(\bb 0, -1)$ to $(\bb 0, +1)$, which (in light of symmetry) is equivalent to $(\bb 0, -1) \in V_0$. Since $G$ is symmetric under $\R$, this will also imply that $\R(V_0) = V_0$.
	
	It remains to show that $(\bb 0, -1) \in V_0$. We do this by explicitly constructing the path from $(\bb 0, +1)$ to $(\bb 0, -1)$. Let $\bb r^{(0)} = \bb r^{(s+1)} = \bb 0$, and for $j = 1,2, \dots, s$ let $\bb r^{(j)} = (r^{(j)}_\omega)$ be given by
$$
	r^{(j)}_\omega = 
	\begin{cases}
		1 & \text{ if } \omega_1 = \omega_2 = \dots = \omega_i = 1,\\
		0 & \text{ otherwise.}
	\end{cases}
$$
	We claim that for each $j = 0,1,\dots,s-1$, there is an edge from $(\bb r^{(j)}, +1)$ to $(\bb r^{(j+1)}, +1)$, and for $j = s$ there is an edge from $(\bb r^{(s)}, +1)$ to $(\bb r^{(s+1)}, -1)$. 
	
	For $j = 0$, define $\bb e^{(0)}$ by $e^{(0)}_{0} = e^{(0)}_{1} = 1$ and $e^{(0)}_{i} = 0$ for $i \neq 0,1$. Direct computation shows that $\delta( \bb r^{(0)}; \bb e^{(0)} )_\omega = \floor{ \frac{1+\omega_1}{2} } = 1$ if $\omega_1 = 1$ and $0$ otherwise. We also have $\abs{ \bb r^{(0)}} = 0$. Hence, $G$ contains an edge from $(\bb r^{(0)},+1)$ to $(\bb r^{(1)},+1)$.
	
	For $1 \leq j \leq s-1$, let $e^{(j)}_{j+1} = 1$ and $e^{(j)}_{i} = 0$ for $i \neq j+1$. We compute that 
	$\delta( \bb r^{(j)}; \bb e^{(j)} )_\omega = \floor{ \frac{r^{(j)}_\omega + \omega_{j+1}}{2} } = 1$ if $\omega_{j+1} = 1$ and $r^{(j)}_\omega = 1$, and $0$ otherwise. Also, $\abs{ \bb r^{(j)} } = 2^{s-j} \equiv 0 \pmod{2}$. Hence, $G$ contains an edge from $(\bb r^{(j)},+1)$ to $(\bb r^{(j+1)},+1)$.
	
	Finally, for $j = s$, let $\bb e^{(s)} = \bb 0$. Computation similar to the one above shows that $G$ contains an edge from $(\bb r^{(s)},+1)$ to $(\bb r^{(s+1)},-1)$.
\end{proof}

\begin{corollary}\label{cor:main-TM}
	There exists a constant $c > 0$ such that $A(L, \bb 0) = O( 2^{ - c  L} )$.
\end{corollary}
\begin{proof}
	Because $G$ is aperiodic and strongly connected, the Frobenius-Perron Theorem implies that there exists a stationary distribution $\pi \colon V_0 \to [0,1]$ such that for each $v,v' \in V_0$ we have $P^{(l)}(v,v') \to \pi(v')$ with exponential convergence rate: 
	\begin{equation}
	\max_{v,v' \in V_0} \abs{  P^{(l)}(v,v') - \pi(v')} = O( 2^{ - c  l} )
	\label{eq:764}
	\end{equation}		
	for some $c > 0$. Because $\R( V_0 ) = V_0$, by symmetry we have $\pi(\R(v)) = \pi(v)$ for all $v \in V_0$. Combining this with \eqref{eq:764}, we obtain
	\begin{equation}
	\max_{v,v'} \abs{  P^{(l)}(v,v') -   P^{(l)}(v,\R(v')) } = O( 2^{ - c  l} ).
	\label{eq:765}
	\end{equation}		
	Using this estimate and the trivial bound $\abs{ A(L,\bb r) } \leq 1$ in \eqref{eq:recurrence-A-TM-II}, we arrive at
	\begin{equation}
	A(L,\bb 0) = O( 2^{ - c  l} + 2^{-(L-l)}).
	\label{eq:766}
	\end{equation}		
	It remains to put $l = L/2$ (say) to conclude that $A(L,\bb 0) = O( 2^{ - c' L})$ with $c' > 0$. 
\end{proof}

\begin{remark}
	Computation of the constant $c$ in the argument above is essentially equivalent to computing the spectral gap for the matrix $\bra{ P(v,v') }_{v,v'\in V_0}$. In particular, for fixed $s$, this is a computationally tractable problem.
\end{remark}

\begin{proof}[Proof of Theorem \ref{thm:TM}]
	Split $[N]$ into intervals $I_j = [m_j2^{L_j}, (m_j+1)2^{L_j})$ where $L_1 > L_2 > \dots$. We then have
	\begin{equation}\label{eq:542}
	\norm{ t }_{U^s[N]} = \norm{ \sum_{j} 1_{I_j} t }_{U^s[N]} \leq \sum_{j} \norm{ 1_{I_j} t }_{U^s[N]}. 
	\end{equation}
	The cubes $\set{ n + \omega \cdot \bb h}{\omega \in \Qs} \subset I_j$ are precisely the translations of the cubes $\set{ n + \omega \cdot \bb h}{\omega \in \Qs} \subset [2^{L_j}]$ by $2^{L_j} m_j$. Since $t(2^{L_j} m_j + n) = t(2^{L_j} m_j)t(n)$ for $n \in [2^{L_j}]$, we conclude that 
	$$\norm{ 1_{I_j} t }_{U^s[N]} = \norm{ 1_{[{I_j}]}}_{U^s[N]} \cdot \norm{ t }_{U^s[2^{L_j}]} \ll { \frac{2^{L_j}}{N} } \norm{  t }_{U^s[2^{L_j}]}.$$
	Inserting this into \eqref{eq:542} and applying Corollary \ref{cor:main-TM} yields (with  $c' = \ffrac{c}{2^s}$ using notation therein) we obtain
	$$\norm{ t }_{U^s[N]} \ll \sum_j 2^{L_j(1-c')}/N \ll N^{-c'}. \qedhere$$
\end{proof} 
\section{Rudin-Shapiro sequence}\label{sec:RS}

We now move on to the Rudin-Shapiro sequence $r(n) = (-1)^{f_{\mathtt{11}}(n)}$, and embark upon the proof of Theorem \ref{thm:RS}. Our argument is similar to the one in Section \ref{sec:TM}, although slightly more technical. Complications arise because of the fact that the $2$-kernel 
$$\cN_{2}(r) = \set{ n \mapsto r(2^l n + m) }{ 0 \leq m < 2^l} = \{ \pm r(n), \pm r(2n+1)\} $$
contains other functions apart from $\pm r(n)$, which forces us to deal with averages more general than those in \eqref{eq:def-A-TM}. 

A key feature of $r(n)$ which allows our argument to work is that $\cN_2(r)$ is symmetric, i.e.\ $\cN_2(r) = -\cN_2(r)$. Denote $\cN_2^+(r) = \{ r(n), r(2n+1) \}$, and fix from now on the value of $s \in \NN_{\geq 2}$. We will study the  averages 
\begin{equation}
	\label{eq:def-A-RS}
	A(L, \bb a, \bb r) := \EE_{n, \bb h }\prod_{\omega \in \Qs} a_\omega ( n + \omega \cdot \bb h + r_\omega),
\end{equation}
where $\set{n + \omega \cdot \bb h}{\omega \in \Qs} \subset [2^L]$, $L \in \NN_0$, $\bb a = (a_\omega)_{\omega \in \Qs}$ with $a_\omega \in \cN_2^+(r)$ and $\bb r = (r_\omega)_{\omega \in \Qs}$ with $r_\omega \in \ZZ$.

We record a recurrence relation analogous to Lemma \ref{lem:recurrence-A-TM}. Now, it will be more convenient to consider $l$ consecutive steps. \begin{lemma}\label{lem:recurrence-A-RS}
	For any $l$, the averages $A(L, \bb a, \bb r)$ obey the recursive relation:
	\begin{equation}
		\label{eq:recurrence-A-RS}
		A(L, \bb a, \bb r) = \EE_{ \bb e} A(L-l, \delta^{(l)}( \bb a; \bb r ,\bb e), \delta^{(l)}( \bb r; \bb e) ) + O(2^{-(L-l)}),
	\end{equation}
	where the average is taken over $\bb e = (e_i)_{i=0}^s \in [2^l]^{s+1}$, $\delta^{(l)}( \bb a; \bb r ,\bb e)$ is given by 
	\begin{equation}
		\label{eq:004a}
		\delta^{(l)}( \bb a; \bb r ,\bb e)_\omega(n) = a_\omega\bra{ 2^l n + \bra{r_\omega + (1,\omega) \cdot \bb e \bmod 2^l } },
	\end{equation}
	 and $ \delta^{(l)}( \bb r; \bb e)$ is given by 
	\begin{equation}
		\label{eq:005b}
	  \delta^{(l)}( \bb r; \bb e)_{\omega} = \floor{ \frac{ r_\omega + (1,\omega) \cdot \bb e  }{2^l} }.
	\end{equation}
\end{lemma}
\begin{proof}
	Fix the value of $l$. Any cube $Q = \set{ n + \omega \cdot \bb h}{ \omega \in \{0,1\}^s}$ can be uniquely written as $\set{ 2^l(n' + \omega \cdot \bb h') + (1,\omega) \cdot \bb e}{ \omega \in \{0,1\}^s}$, where $\bb e \in [2^{l}]^{s+1}$. Up to errors of the order of $O(2^{-(L-l)})$, choosing $Q \subset [2^L]$ uniformly at random is equivalent to choosing $\bb e \in [2^{l}]^{s+1}$ and $Q' = \set{ n' + \omega \cdot \bb h'}{ \omega \in \{0,1\}^s} \subset [2^{L-l}]$ uniformly at random, hence
	\begin{align*}
		A(L, \bb a, \bb r) &= 
		\EE_{ \bb e } \EE_{n', \bb h' } \prod_{\omega \in \Qs} a_\omega( 2^k( n' + \omega \cdot \bb h') + (1,\omega) \cdot \bb e + r_\omega ) + O(2^{-(L-l)}) 
		\\& = \EE_{ \bb e } \EE_{n', \bb h' } \prod_{\omega \in \Qs} \delta^{(l)}(\bb a; \bb r, \bb e)_\omega ( n' + \omega \cdot \bb h' + \delta^{(l)}(\bb r; \bb e)_\omega)  + O(2^{-(L-l)}),
	\end{align*}
	which is precisely the stated formula.
\end{proof}
For $l = 1$, we write $\delta$ for $\delta^{(1)}$. Note that the symbol $\delta$ is used to denote several different functions, but this will not lead to ambiguity because it can always be inferred from the arguments which function is meant.

Like in the previous section, we introduce a random walk $\cW_{\mathrm{RS}}$ on a graph $G = (V,E)$, which is associate to the averages $A(L,\bb a, \bb r)$. The set of vertices $V$ consists of triples $(\bb a, \bb r, \pm 1)$, where $a_\omega \in \cN^+_2(r)$ and $r_\omega \in \ZZ$. For $v = (\bb a, \bb r, \sigma) \in V$, we write $A(L,v) = \sigma A(L,\bb a, \bb r)$.

Using the fact that $\cN_2(r) = \cN_2^+(r) \cup (-\cN_2^+(r))$, we see that for any $\bb a$ with $a_\omega \in \cN_2(r)$, we can find $\barbb a = (\bar a_\omega)_{\omega \in \Qs}$ with $\bar a_\omega \in \cN_2^+(r)$ and $\sigma = \pm 1$, such that $\prod_{\omega} a_\omega(x_\omega) = \sigma \prod_{\omega} \bar a_\omega(x_\omega)$. In particular, for any $\bb r$ we have $A(L,\bb a, \bb r) = A(L, v)$ for any $L$, where $v = (\barbb a, \bb r, \sigma) \in V$.

Let $l \in \NN$ be fixed. Then, for $\bb e \in [2^l]^{s+1}$ and $v  = (\bb a, \bb r, \sigma) \in V$, we let $\delta^{(l)}(v;\bb e) \in V$ denote the vertex constructed above, corresponding to the averages $ A(L, \delta^{(l)}( \bb a; \bb r ,\bb e), \delta^{(l)}(\bb r; \bb e) )$. (In other words, $\delta^{(l)}(v;\bb e) = (\bb a', \bb r', \sigma')$, where $a'_\omega = \pm \delta^{(l)}(\bb a; \bb r, \bb e)_\omega$, $r'_\omega = \delta^{(l)}( \bb r; \bb e)$ and $\sigma' = \pm 1$ is chosen so that $\sigma' \prod_\omega a_\omega(x_\omega) = \sigma \prod_\omega \delta^{(l)}(\bb a; \bb r, \bb e)_\omega(x_\omega)$).

The transition probabilities are given by $P(v,v') = \PP_{\bb e \in \{0,1\}^{s+1}}( \delta(v; \bb e) = v')$ for $v,v' \in V$, so that \eqref{eq:recurrence-A-RS} for $l = 1$ is equivalent to
	\begin{equation}
		\label{eq:recurrence-A-RS-II}
		A(L, v) = \sum_{v' \in V}
		 P\bra{v, v' } A(L -1, v') + O(2^{-L}).
	\end{equation}

The edge from $v$ to $v'$ is present in the edge set $E$ of $G$ if $P(v,v') > 0$. 

More generally, for arbitrary $l \in \NN$ we have 
\begin{equation}
		\label{eq:recurrence-A-RS-III}
A(L, v) = \sum_{v' \in V}  P^{(l)}\bra{v, v' } A(L - l, v') + O(2^{-(L-l)}),
	\end{equation}
where $P^{(l)}\bra{v, v' } = \PP_{\bb e \in [2^l]^{s+1}}( \delta^{(l)}(v; \bb e) = v')$ denotes the probability of transition from $v$ to $v'$ in $l$ steps. Accordingly, a path of length $l$ from $v = (\bb a, \bb r, +1)$ to $v' \in V$
  exists if and only if the average $A(L-l,v')$ is present on the right hand side of \eqref{eq:recurrence-A-RS}, meaning that there exists $\bb e = (e_i)_{i=0}^k$, $0 \leq e_i < 2^l$, such that $\delta^{(l)}( (\bb a, \bb r, +1); \bb e) = v'$. 
Following the same reasoning as before, we note that $G$ has a natural symmetry $\R \colon V \to V$ given by $(\bb a, \bb r, \sigma) \mapsto (\bb a, \bb r, -\sigma)$, which preserves the transition probabilities. We will denote by $V_0$ the set of vertices reachable from the initial vertex $v_0 = ((r)_{\omega \in \Qs}, \bb 0, +1)$ and by $G_0$ the induced graph. 

\begin{proposition}\label{lem:G-RS-connected}
	Let $G_0$ be the graph constructed above. Then $G_0$ is finite, strongly connected, aperiodic, and preserved by $\R$.
\end{proposition}
\begin{proof}
	Finiteness, aperiodicity, and strong connectedness follow from essentially the same argument as in Propositions \ref{lem:G-TM-connected}. It remains to prove that $\R (v_0)$ is reachable from $v_0$.  
	
	Pick any $l \geq s+2$, and $e_i = 2^{i-1}$ for $i = 1,2,\dots, s$; we leave $0 \leq e_0 < 2^{s}$ undefined for the time being. It follows from Lemma \ref{lem:recurrence-A-RS} and subsequent discussion that $G_0$ contains a path of length $l$ from $v_0$ to $v_1 = (\bb a, \bb r, \sigma)$ equal to $\delta^{(l)}(v_0; \bb e)$.
		
	We will now identify the vertex $v_1$. As for $\bb a = (a_\omega)_\omega$, we notice that
	\begin{equation}
		\label{eq:761}
			a_\omega(n) = \pm r(2^l n + (1, \omega) \cdot \bb e ) = \pm r(n) r((1, \omega) \cdot \bb e) = r(n).
	\end{equation}
	Above, we use that fact that $(1, \omega) \cdot \bb e < 2^{l-1}$. Next, $\bb r = (r_\omega)_\omega$ is given by
	$$
		r_\omega = \floor{ \frac { (1, \omega) \cdot \bb e }{2^l} } = 0,
	$$	
	by virtue of the same estimate as before. Finally, $\sigma$ is the product of the $\pm 1$ factors implicit in \eqref{eq:761}, hence
	$$
		\sigma = \prod_{\omega \in \Qs} r((1, \omega) \cdot \bb e) = \prod_{m=0}^{2^s-1} r( m + e_0).
	$$
	Thus, $v_1$ is equal to $\R (v_0)$, provided that $\sigma = \sigma(e_0)$ defined above is equal to $-1$, and $v_1 = v_0$ otherwise. It remains to find $e_0$ for which $\sigma(e_0) = -1$. In fact, it will suffice to show that $\sigma(e_0)$ is not constant with respect to $e_0$. Since $\sigma(e_0 + 1)/\sigma(e_0) = r(2^s + e_0)/r(e_0)$, we have $\sigma(e_0+1) = - \sigma(e_0)$ for any choice $2^{s-1} \leq e_0 < 2^s$, which finishes the argument.
\end{proof}

\begin{corollary}\label{cor:main-RS}
	There exists a constant $c > 0$ such that $\norm{t}_{U^s[2^L]} = O(2^{-cL})$. 
\end{corollary}
\begin{proof}
	Direct adaptation of the argument in Corollary \ref{cor:main-TM}.
\end{proof}

\begin{proof}[Proof of Theorem \ref{thm:RS}]
	We begin by splitting the interval $[N]$ into a disjoint union of intervals $I_j$ of the form $[ m_j 2^{L_j+1},  m_j 2^{L_j+1} + 2^{L_j}),$ and a remainder part $J$ so that $\abs{J} \ll \log N$ and each exponent $L$ appears among $L_j$'s at most $\log N$ times. This can be accomplished by first splitting $[N]$ into dyadic intervals $[m'_j 2^{L'_j}, (m'_j+1) 2^{L'_j})$ (with $L_j'$ all distinct), and then splitting each of these further into intervals $[(m'_j+1) 2^{L'_j} - 2^{k}, (m'_j+1) 2^{L'_j} - 2^{k} + 2^{k-2})$ for $2 \leq k < L'_j$, and the singleton of $(m'_j+1) 2^{L'_j} - 1$ (which we put in $J$).

	For each of the intervals $I_j$ and for any $n + m_j2^{L_j+1} \in I_j$, $n \in [2^j]$, we have $r(n + m_j2^{L_j+1}) = r(n) r(m_j)$, whence $\norm{ 1_{I_j} r }_{U^s[N]} \ll { \frac{2^{L_j}}{N} } \norm{ r }_{U^s[2^{L_j}]}$. Using Corollary \ref{cor:main-RS} and the bound $\norm{1_J}_{U^s[N]} \ll \bra{ \abs{J}/N }^{1/2^s}$, we may now estimate:
	\begin{align*}
	\norm{ t }_{U^s[N]} &\ll \sum_{j} \frac{2^{L_j(1-c)}}{N} + N^{-\ffrac{1}{2^s} + o(1)}
	\\ & \ll \frac{\log N}{N} \sum_{L \leq \log N} 2^{L(1-c)} + N^{-\ffrac{1}{2^s} + o(1)} \ll N^{-c'},
	\end{align*}
	where $0 < c' < \min(c,1/2^s)$ is arbitrary.
\end{proof} 
\section{Closing remarks}\label{sec:End}

Our argument in Section \ref{sec:RS} dealing with the Rudin-Shapiro sequence can be generalised to other automatic sequences. We hope to address this in an upcoming paper with Jakub Byszewski and Clemens M\"{u}llner. Here, we discuss  some simple generalisations which can be obtained by slight adaptations of the existing argument, as well as the key obstacles which need to be overcome for further progress to be made.

A crucial feature of the Rudin-Shapiro sequence which we exploited was the symmetry of the kernel, so for the time being let us restrict our attention to $2$-automatic sequences $a(n)$ with $\cN_2(a) = - \cN_2(a)$. Natural examples of such sequences are the given by the ``pattern counting'' sequences of the form $a(n) = (-1)^{f_{\pi}(n)}$, where $\pi$ is a word over the alphabet $\{\mathtt{1},\mathtt{0}\}$ and $f_\pi(n)$ denotes the number of times $\pi$ appears in the binary expansion of $n$. (Hence, $\pi = \mathtt{1}$ for Thue-Morse and $\pi = \mathtt{11}$ for Rudin-Shapiro. To avoid technical complications, assume that $\pi$ begins and ends with $\mathtt{1}$.)

The recursive relation analogous to \eqref{eq:recurrence-A-RS} from Lemma \ref{lem:recurrence-A-RS} holds in full generality, and similarly the corresponding random walk can be constructed without any significant modifications. It remains true that the underlying graph is symmetric, and that the analogue of \eqref{eq:recurrence-A-RS-III} holds. Provided that the graph has the properties mentioned in Proposition \ref{lem:G-RS-connected}, the analogue of Corollary \ref{cor:main-RS} stating that $\norm{a}_{U^s[2^L]} = O(2^{-cL})$ follows immediately. To obtain the bound $\norm{a}_{U^s[N]} = O(N^{-c})$ for general $N$, one can apply a decomposition of $[N]$ analogous to that in the Proof of Theorem B. For pattern counting sequences introduced above, this construction is repeated almost verbatim, except one uses intervals of the form $[2^L(2^{\abs{\pi}} m),2^L(2^{\abs{\pi}} m+1))$, where $\abs{\pi}$ denotes the length of the pattern.

The key difficulty lies in proving the analogue of Proposition \ref{lem:G-RS-connected}, asserting that the graph supported on the vertices reachable from the origin in the random walk is finite, aperiodic, strongly connected and preserved under the natural symmetry. Finiteness is clear in full generality, by the same reasoning as in Proposition \ref{lem:G-TM-connected}. We expect that aperiodicity should be easy to check; for pattern counting sequences we simply have a loop labelled $\bb 0$ at the origin. Strong connectedness does not appear to be crucial, since we may consider the strongly connected components independently; for pattern counting sequences continuing from any vertex along the edges labelled $\bb 0$ leads to either the origin or its symmetric image.

Proving the symmetry is presumably the most difficult part. However, in the case of pattern counting sequences, it can be carried out by a slight modification of the argument in Proposition \ref{lem:G-RS-connected}. The only difference is that (with the notation taken from the proof of Proposition \ref{lem:G-RS-connected}) we need to take $l \geq s + \abs{\pi}$ now, and notice that $\sigma(e_0+1)/\sigma(e_0) = a(2^s + e_0)/a(e_0) = -1$ for suitable choice of $e_0$. Hence, we expect that for a pattern counting sequence such $a(n) = (-1)^{f_{\pi}(n)}$ we have $\norm{ a }_{U^s[N]} = O(N^{-c})$ for a constant $c = c(a,s) > 0$.

Conversely, one may ask about the minimal conditions under which it can be shown that $\norm{ a }_{U^s[N]} = o(1)$. It is well-known that a sequence with small uniformity norms cannot correlate with a polynomial phase, or indeed with an $s$-step nilsequence. While it would be surprising to find an automatic sequence correlating with (say) a quadratic phase $e^{2 \pi i \alpha n^2}$ (with $\alpha \in \RR \setminus \QQ$), it is certainly possible to have automatic sequences which correlate with periodic sequences. (In fact, in the situation above it is not hard to show that $\EE_{n < N} a(n) e^{2 \pi i \alpha n^2} = o(1)$ as $N \to \infty$.)

Motivated by our main theorems and the above discussion, we are led to suspect the following. The suspicion is especially strong in the case of sequences with symmetric kernel.

\begin{conjecture*}
	Let $a(n)$ be a $2$-automatic sequence 	such that $\EE_{n < N} a(qn + r) \to 0$ as $N \to \infty$ for any $q \in \NN,\ r \in \NN_0$. Then, $\norm{a}_{U^s[N]} \to 0$ as $N \to \infty$ for any $s \in \NN$.
\end{conjecture*}

\bibliographystyle{alpha}
\bibliography{bibliography}

\begin{thebibliography}{DDM12}

\bibitem[AS03]{AS}
Jean-Paul Allouche and Jeffrey Shallit.
\newblock {\em Automatic sequences}.
\newblock Cambridge University Press, Cambridge, 2003.
\newblock Theory, applications, generalizations.

\bibitem[DDM12]{DeshouillersDrmotaMorgenbesser-2012}
Jean-Marc Deshouillers, Michael Drmota, and Johannes~F. Morgenbesser.
\newblock Subsequences of automatic sequences indexed by {$\lfloor n^c\rfloor$}
  and correlations.
\newblock {\em J. Number Theory}, 132(9):1837--1866, 2012.

\bibitem[DMR13]{DrmotaMauduitRivat-TM-squares}
Michael Drmota, Christian Mauduit, and Jo\"el Rivat.
\newblock The {T}hue-{M}orse sequence along squares is normal, 2013.
\newblock Preprint.

\bibitem[Drm14]{Drmota-2014}
Michael Drmota.
\newblock Subsequences of automatic sequences and uniform distribution.
\newblock In {\em Uniform distribution and quasi-{M}onte {C}arlo methods},
  volume~15 of {\em Radon Ser. Comput. Appl. Math.}, pages 87--104. De Gruyter,
  Berlin, 2014.

\bibitem[Fan]{Fan-A}
Aihua Fan.
\newblock Oscillating sequences of higher order and topological systems of
  quasi-discrete spectrum.
\newblock Preprint.

\bibitem[Gel68]{Gelfond-1967}
A.~O. Gel'fond.
\newblock Sur les nombres qui ont des propri\'et\'es additives et
  multiplicatives donn\'ees.
\newblock {\em Acta Arith.}, 13:259--265, 1967/1968.

\bibitem[Gow01]{Gowers-2001}
W.~T. Gowers.
\newblock A new proof of {S}zemer\'edi's theorem.
\newblock {\em Geom. Funct. Anal.}, 11(3):465--588, 2001.

\bibitem[Gre]{Green-book}
Ben Green.
\newblock {\em {Higher-Order Fourier Analysis, I}}.
\newblock (Notes available from the author).

\bibitem[HK09]{HostKra-2009}
Bernard Host and Bryna Kra.
\newblock Uniformity seminorms on {$\ell^\infty$} and applications.
\newblock {\em J. Anal. Math.}, 108:219--276, 2009.

\bibitem[MR09]{MauduitRivat-2009}
Christian Mauduit and Jo\"el Rivat.
\newblock La somme des chiffres des carr\'es.
\newblock {\em Acta Math.}, 203(1):107--148, 2009.

\bibitem[MR10]{MauduitRivat-2010}
Christian Mauduit and Jo\"el Rivat.
\newblock Sur un probl\`eme de {G}elfond: la somme des chiffres des nombres
  premiers.
\newblock {\em Ann. of Math. (2)}, 171(3):1591--1646, 2010.

\bibitem[MR15]{MauduitRivat-2015}
Christian Mauduit and Jo\"el Rivat.
\newblock Prime numbers along {R}udin-{S}hapiro sequences.
\newblock {\em J. Eur. Math. Soc. (JEMS)}, 17(10):2595--2642, 2015.

\bibitem[MS98]{Mauduit-1998}
Christian Mauduit and Andr{\'a}s S{\'a}rk{\"o}zy.
\newblock On finite pseudorandom binary sequences. {II}. {T}he {C}hampernowne,
  {R}udin-{S}hapiro, and {T}hue-{M}orse sequences, a further construction.
\newblock {\em J. Number Theory}, 73(2):256--276, 1998.

\bibitem[MS15]{MullnerSpiegelhofer}
Clemens M\"{u}llner and Lukas Spiegelhofer.
\newblock Normality of the {T}hue-{M}orse sequence along {P}iatetski-{S}hapiro
  sequences, {II}, 2015.
\newblock Preprint. {\href{ https://arxiv.org/abs/1511.01671 }{arXiv:1511.01671
  [math.NT]}}.

\bibitem[Tao12]{Tao-book}
Terence Tao.
\newblock {\em Higher order {F}ourier analysis}, volume 142 of {\em Graduate
  Studies in Mathematics}.
\newblock American Mathematical Society, Providence, RI, 2012.

\end{thebibliography}

\end{document}